\documentclass[12pt]{article}
\usepackage{amsthm}
\usepackage[margin=2cm]{geometry}
\usepackage[fleqn]{amsmath}
\usepackage{amssymb}
\usepackage{enumitem}
\usepackage{tikz}
\usetikzlibrary{decorations.markings}

\newtheorem{theorem}{Theorem}

\newtheorem{proposition}{Proposition}
\newtheorem{corollary}{Corollary}

\DeclareMathOperator{\Cay}{Cay}
\DeclareMathOperator{\Aut}{Aut}

\begin{document}
\title{Mixed Moore Cayley graphs}
\author{Grahame Erskine\\{\small Open University, Milton Keynes, UK}\\ \texttt{\small grahame.erskine@open.ac.uk}}

\date{}

\maketitle
\let\thefootnote\relax\footnote{Mathematics subject classification: 05C25, 05C35}
\let\thefootnote\relax\footnote{Keywords: degree-diameter problem, mixed graphs, Moore graphs}
\begin{abstract}\noindent
The degree-diameter problem seeks to find the largest possible number of vertices in a graph having given diameter and given maximum degree.
There has been much recent interest in the problem for mixed graphs, where we allow both undirected edges and directed arcs in the graph.
For a diameter 2 graph with maximum undirected degree $r$ and directed out-degree $z$, 
a straightforward counting argument yields an upper bound $M(z,r,2)=(z+r)^2+z+1$ for the order of the graph. 
Apart from the case $r=1$, the only three known examples of mixed graphs attaining this bound are Cayley graphs,
and there are an infinite number of feasible pairs $(r,z)$ where the existence of mixed Moore graphs with these parameters is unknown.
We use a combination of elementary group-theoretical arguments and computational techniques to rule out the existence of further examples of mixed Cayley graphs
attaining the Moore bound for all orders up to 485.
\end{abstract}

\tikzset{middlearrow/.style={
        decoration={markings,
            mark= at position 0.9 with {\arrow[scale=2]{#1}} ,
        },
        postaction={decorate}
    }
}

%----------------------------------------------
\section{Preliminaries}
The degree-diameter problem has its roots in the efficient design of interconnection networks.
We try to find the maximum possible number of vertices in a graph where we constrain both the largest degree of any vertex and the diameter of the graph.
For more information on the history and development of the degree-diameter problem, see the survey~\cite{miller2005moore}.
The degree-diameter problem is typically studied in both the undirected and directed cases.
Recently, there has been much interest in the problem as it is applied to \emph{mixed} graphs, where we allow both undirected edges and directed arcs in the graph.

In the undirected case, an upper bound for the largest possible order of a graph of maximum degree $d>2$ and diameter $k>1$ 
is the \emph{Moore bound}~\cite{miller2005moore}:
\[M(d,k)=1+d\frac{(d-1)^k-1}{d-2}\]
A graph attaining this bound is known as a \emph{Moore graph}. 
It is known~\cite{Hoffman1960,Bannai1973} that such a graph must have diameter 2 and degree $d\in \{2,3,7,57\}$,
with existence of the graph corresponding to $d=57$ being a famous open problem.

For digraphs, the Moore bound for graphs of maximum out-degree $d>1$ and diameter $k>1$ has an even simpler form:
\[M(d,k)=\frac{d^{k+1}-1}{d-1}\]
It is well-known~\cite{Plesnik1974,Bridges1980} that no Moore digraphs of diameter greater than one exist apart from the directed 3-cycle.

In this paper we concentrate on the case of \emph{mixed} graphs where we allow both undirected edges and directed arcs in our graphs.
We can view the case of mixed graphs either as a generalisation of the undirected case (allowing arcs as well as edges) or as a specialisation of the directed case
(where we insist that a number of the arcs must be present with their reverses).
We adopt the usual notation in the literature. The maximum undirected degree of any vertex is $r$, and the maximum directed out-degree is $z$.
The general expression for the Moore bound for mixed graphs is rather more awkward in closed form~\cite{Buset2016}.
However, Nguyen, Miller and Gimbert~\cite{Nguyen2007} showed in 2007 that no mixed Moore graph can exist for diameters greater than 2.
We therefore concentrate on the diameter 2 case where it is straightforward to show that the Moore bound is :
\[M(z,r,2)=(z+r)^2+z+1\]

In 1979, Bos\'ak~\cite{Bosak1979} derived (using a modification of the spectral method used by Hoffman and Singleton~\cite{Hoffman1960} in the undirected case)
a numerical constraint on the sets of parameters $(r,z)$ for which a mixed Moore graph of diameter 2 can exist.
Bos\'ak's condition is that $r=(c^2+3)/4$ for some odd integer $c$ dividing $(4z-3)(4z+5)$.
In contrast with the undirected case, this condition means that there are an infinite number of feasible pairs $(r,z)$.

We will concentrate on the restricted problem of mixed \emph{Cayley graphs}.
Given a finite group $G$ and a subset $S\subseteq G\setminus\{1\}$, we define the Cayley graph $\Cay(G,S)$ to have vertex set $G$
and arcs from a vertex $g$ to $gs$ for every $s\in S$. If both $s$ and $s^{-1}$ are in $S$, then for every $g$ we have arcs from $g$ to $gs$ and $gs$ to $g$
which we view as an undirected edge between $g$ and $gs$.
Thus if $S=S_1\cup S_2$ where $S_1=S_1^{-1}$ and $S_2\cap S_2^{-1}=\emptyset$, then $\Cay(G,S)$ is a mixed graph of undirected degree $r=|S_1|$
and directed degree $z=|S_2|$.

It is easy to see that a Cayley graph $\Cay(G,S)$ has diameter at most 2 if and only if every element of $G$ can be expressed as a product of at
most 2 elements of $S$. Thus we relate the properties of the graph to properties of the group $G$ and subset $S$.

%----------------------------------------------
\section{Known Moore graphs}

We can see (Table~\ref{tab:moorefeas}) how Bos\'ak's condition means that the range of feasible pairs $(r,z)$ for which a Moore graph can exist is quite limited.
Nevertheless, there are infinitely many pairs $(r,z)$ for which the existence of a mixed Moore graph (Cayley or otherwise) is not known.
In fact, almost all the feasible pairs remain open as to the existence or otherwise of a Moore graph.

\begin{table}\centering
\begin{tabular}{|ccc|}
\hline
Undirected & Directed & Order \\
degree $r$ & degree $z$ & $n$\\
\hline
1 & any & $r^2+2r+3$\\
\hline
3 & 1 & 18\\
  & 3 & 40\\
	& 4 & 54\\
	& 6 & 88\\
	& 7 & 108\\
	& $\cdots$ & $\cdots$\\
\hline
7 & 2 & 84\\
	& 5 & 150\\
	& 7 & 204\\
	& $\cdots$ & $\cdots$\\
\hline
13 & 4 & 294\\
   & 6 & 368\\
	& $\cdots$ & $\cdots$\\
\hline
21 & 1 & 486\\
	& $\cdots$ & $\cdots$\\
\hline
$\cdots$& $\cdots$ & $\cdots$\\
\hline
\end{tabular}
\caption{Feasible values for mixed Moore graphs}
\label{tab:moorefeas}
\end{table}

In the case $r=1$, it is immediate that any positive integer $z$ yields a feasible pair, and indeed such Moore graphs always exist by the following construction.
The \emph{Kautz digraphs} $Ka(d,2)$ are a family of mixed Moore graphs of diameter 2, directed degree $z=d-1$ and undirected degree $r=1$.
The vertices are the words $ab$ of length 2 over an alphabet of $d+1$ letters where we insist $a\neq b$. So there are $d(d+1)=(z+r)^2+z+1$ vertices.
There is a directed edge from $ab$ to $bc$ for all of the $d$ eligible values of $c$. The edge from $ab$ to $ba$ can be considered as the
undirected edge since the reverse edge also exists. All other edges from $ab$ are purely directed.

The graph has diameter 2 since there is a path $ab\to xy$ of length 1 if $x=b$ and $ab\to bx \to xy$ of length 2 if $x\neq b$.
An example in the case $d=2$ is shown in Figure~\ref{fig:kautz}.

\begin{figure}\centering

\begin{tikzpicture}[x=0.2mm,y=-0.2mm,inner sep=0.2mm,scale=0.7,thick,vertex/.style={circle,draw,minimum size=10,fill=lightgray}]
\small
\node at (380,440) [vertex] (v1) {$ac$};
\node at (460,300) [vertex] (v2) {$cb$};
\node at (380,580) [vertex] (v3) {$ca$};
\node at (300,300) [vertex] (v4) {$ba$};
\node at (580,240) [vertex] (v5) {$bc$};
\node at (180,240) [vertex] (v6) {$ab$};
\path
	(v1) edge[->=stealth] (v2)
	(v1) edge[ultra thick] (v3)
	(v2) edge[->=stealth] (v4)
	(v2) edge[ultra thick] (v5)
	(v5) edge[->=stealth] (v3)
	(v4) edge[->=stealth] (v1)
	(v4) edge[ultra thick] (v6)
	(v6) edge[->=stealth] (v5)
	(v3) edge[->=stealth] (v6)
	;
\end{tikzpicture}
\caption{The Kautz digraph $Ka(2,2)$}
\label{fig:kautz}
\end{figure}
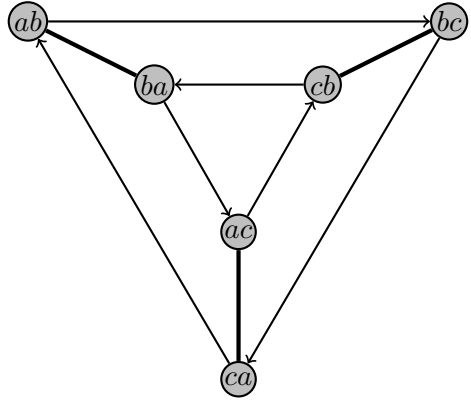

The Kautz digraphs $Ka(d,2)$ are not Cayley graphs for all values of $d$, and in fact they turn out to be Cayley graphs precisely when $d+1$ is a prime power
(see for example~\cite{Brunat1995}).
Until very recently, these graphs and a single further example of Bos\'ak with parameters $r=3,z=1$ (and hence order 18) were the only known mixed Moore graphs.
(See the survey paper~\cite{miller2005moore} for more on these known graphs.)

Recently, J\o rgensen~\cite{Joergensen2015} has reported a pair of graphs with $r=3,z=7$ and hence order 108.
These graphs are interesting because they are Cayley graphs (as indeed is Bos\'ak's graph of order 18).
The two graphs are in fact a transpose pair, where one graph is obtained from the other by reversing the direction of the directed arcs.

On the negative side, no simple combinatorial argument has yet been found to rule out any feasible parameter pairs satisfying Bos\'ak's condition.
For small graphs, an exhaustive computational approach is now becoming feasible with advances in CPU power and algorithms.
L\'opez, P\'erez-Ros\'es and Pujol\`as~\cite{Lopez2014} ruled out the existence of mixed Moore Cayley graphs of orders 40 and 54.
Recently, L\'opez,  Miret and Fern\'andez~\cite{Lopez2016} have used computational techniques to show that 
there are no mixed Moore graphs (Cayley or otherwise) at orders 40, 54 or 84.

\section{Searching for new examples}
It seems unlikely that brute-force exhaustive search algorithms will take us much further in the table.
Inspired by J\o rgensen's result and the fact that the Bos\'ak graph of order 18 is also a Cayley graph,
we describe a search algorithm to extend the work of L\'opez, P\'erez-Ros\'es and Pujol\`as~\cite{Lopez2014} 
to look for further examples of mixed Moore Cayley graphs.

Given a feasible pair $(r,z)$, we wish to find a group $G$ and a set $S\subseteq G$ such that the graph $\Cay(G,S)$ has order $n=(z+r)^2+z+1$, 
undirected degree $r$, directed degree $z$ and diameter 2.
For ease of explanation we split $S$ into the undirected generators $S_1$ and the directed generators $S_2$.
Then $|S_1|=r,|S_2|=z,S_1=S_1^{-1},S_2\cap S_2^{-1}=\emptyset$.

The naive approach of simply testing all possible sets $S$ very quickly becomes computationally infeasible.
Our strategy therefore is to look for properties of Moore graphs and corresponding properties of Cayley graphs which will allow us to reduce the search space.
We begin with some elementary yet useful properties of mixed Moore graphs.
\begin{proposition}\label{prop:mmprops}
Let $\Gamma$ be a mixed Moore graph of diameter 2, undirected degree $r$ and directed degree $z$.
\begin{enumerate}[label=(\roman*),topsep=0pt,itemsep=0.5ex]
	\item If $u,v\in V(\Gamma)$ are distinct vertices then there is one and only one path of length 1 or 2 from $u$ to $v$.
	\item $\Gamma$ contains no undirected cycle of length 3 or 4.
	\item $\Gamma$ is totally regular (all vertices have the same directed in-degree and out-degree $z$).
	\item Every arc in $\Gamma$ is contained in exactly one directed 3-cycle.
\end{enumerate}
\end{proposition}
\begin{proof}
Item \emph{(i)} follows immediately from the counting argument deriving the Moore bound by considering the spanning tree of $\Gamma$ rooted at $u$.
Item \emph{(ii)} is a consequence of \emph{(i)}.
Item \emph{(iii)} was proved by Bos\'ak~\cite{Bosak1979}.

To see why \emph{(iv)} is true, consider a vertex $u\in V(\Gamma)$. Then $u$ has $z$ directed out-neighbours $v_1,\ldots,v_z$.
Since $\Gamma$ is totally regular, $u$ must have $z$ directed in-neighbours $w_1,\ldots w_z$.
These cannot be at distance 1 from $u$, so each $w_i$ is reached by a path of length 2 from $u$.
There can be no undirected edges in any of these paths, since that would lead to the end vertices of such an edge violating \emph{(i)}.
So each $w_i$ is reached by a directed path of length 2 from $u$ passing through some $v_j$.
These $v_j$ must be distinct, since if any were repeated it would have two paths of length 2 to $u$.
Thus every arc $u\to v_j$ emanating from $u$ lies in the unique directed triangle $u\to v_j\to w_i\to u$.
\end{proof}

Now we can use these properties to develop constraints on our generating set $S=S_1\cup S_2$ to narrow the search for mixed Moore Cayley graphs.

\begin{proposition}\label{prop:mmcayprops}
Let $\Gamma$ be a mixed Moore graph of diameter 2, undirected degree $r$ and directed degree $z$.
Suppose that $\Gamma\cong\Cay(G,S)$ where $G$ is a group of order $n=(r+z)^2+z+1$ and the generating set $S$ consists of undirected generators $S_1$ and
directed generators $S_2$. Then:
\begin{enumerate}[label=(\roman*),topsep=0pt,itemsep=0.5ex]
	\item No element of $S_1$ has order 3 or 4.
	\item No element of $S_2$ is an involution.
	\item No pair of elements in $S_1$ has a product of order 2.
	\item No two distinct elements of $S$ commute, apart from the inverse pairs in $S_1$.
	\item $S$ is product-free (that is, $S\cap SS=\emptyset$).
	\item All non-identity products of two elements of $S$ are unique.
	\item The elements of $S_2$ are of two types:
		\begin{enumerate}[label=\arabic*.]
			\item Elements of order 3
			\item Triples of distinct elements $\{a,b,c\}$, each of order at least 4, such that $(ab)^{-1}=c$
		\end{enumerate}
\end{enumerate}
\end{proposition}
\begin{proof}
These facts follow immediately from the properties of the graph:
\begin{enumerate}[label=(\roman*),topsep=0pt,itemsep=0.5ex]
	\item follows from Proposition~\ref{prop:mmprops}\emph{(ii)}
	\item is immediate because $S_2$ is inverse-free.
	\item is true because such a pair would lead to an undirected 4-cycle.
	\item follows from Proposition~\ref{prop:mmprops}\emph{(i)}.
	\item follows from Proposition~\ref{prop:mmprops}\emph{(i)}.
	\item follows from Proposition~\ref{prop:mmprops}\emph{(i)}.
	\item follows from Proposition~\ref{prop:mmprops}\emph{(iv)}.\qedhere
\end{enumerate}
\end{proof}

We note that the conditions of Proposition~\ref{prop:mmcayprops}(v) and (vi) must also hold for any subset of $S$. This motivates the following definition.

Let $T\subseteq G$ with $T=T_1\cup T_2,T_1=T_1^{-1},T_2\cap T_2^{-1}=\emptyset,|T_1|=r',|T_2|=z'$.
Define $P(T)=|\{1\}\cup T\cup TT|$.
We say $T$ is a \emph{feasible} subset of generators if $P(T)=(z'+r')^2+z'+1$.

We have two further ways to reduce the number of sets $S$ we need to search for a given group $G$.
The first is the well-known result that if $\phi$ is an automorphism of the group $G$, then $\Cay(G,S)\cong\Cay(G,\phi(S))$.
So we need not consider all possible sets -- only orbit representatives under the action of $\Aut(G)$.

The second idea is to exploit the fact that all mixed Moore graphs must have even order (a consequence of Bos\'ak's condition).
So a suitable group $G$ for a Cayley graph must have even order, and may in many cases have an index 2 subgroup.
\begin{proposition}\label{prop:mmcayidx2}
Let $\Gamma$ be a mixed Moore graph of diameter 2, undirected degree $r$ and directed degree $z$.
Suppose that $\Gamma\cong\Cay(G,S)$ where $G$ is a group of order $n=(z+r)^2+z+1$ and the generating set $S$ consists of undirected generators $S_1$ and
directed generators $S_2$. Suppose further that $G$ admits an index 2 subgroup $H$ and that $|S_1\cap H|=s_1$ and $|S_2\cap H|=s_2$. Then:
\[s_1+s_2=\frac{2(z+r)-1\pm\sqrt{4r-3}}{4}\]
\end{proposition}
\begin{proof}
We know each non-identity element of $H$ can be expressed uniquely as a product of 1 or 2 elements of $S$. We count these products.
Firstly, there are $s_1+s_2$ elements of $S\cap H$.
Any other element is either a product of 2 elements of $S\cap H$ or 2 elements of $S\cap(G\setminus H)$.
In the first case there are $s_1(s_1-1)+2s_1s_2+s_2^2$ possibilities.
In the second case there are $(r-s_1)(r-s_1-1)+2(r-s_1)(z-s_2)+(z-s_2)^2$.
Writing $s=s_1+s_2$ we see following some manipulation that the total number of elements of $H$ which we can write as a product of 0, 1 or 2 elements of $S$ is
$2s^2+s(1-2(z+r))+(z+r)^2-r+1$.
But $H$ is an index 2 subgroup and so contains exactly $((z+r)^2+z+1)/2$ elements.
Solving this quadratic equation for $s$ yields the stated result.
\end{proof}
It might be thought that this provides a very strong condition, since the expression for $s_1+s_2$ must clearly give an integer result.
However, it is interesting that Bos\'ak's condition on allowable values of $r,z$ means that this expression always gives one integer solution for $s_1+s_2$.
Nevertheless, the condition does give a useful way to cut down the search space when we have an index 2 subgroup $H$, since it precisely determines the overall
split of generators between $H$ and $G\setminus H$.
In addition, we have a useful corollary allowing us to exclude some groups from consideration entirely.
\begin{corollary}
Suppose $\Gamma$ and $G$ are as in the statement of Proposition~\ref{prop:mmcayidx2}.
Then if $2(z+r)-\sqrt{4r-3}>9$ then $G$ cannot contain an index 2 abelian subgroup $H$. 
\end{corollary}
\begin{proof}
If $H$ is an index 2 abelian subgroup of $G$, then if $2(z+r)-\sqrt{4r-3}>9$,
by Proposition~\ref{prop:mmcayidx2} the generating set $S$ contains more than 2 elements of $H$.
This is contrary to Proposition~\ref{prop:mmcayprops}(iv).
\end{proof}

We can now describe the search algorithm. Given a feasible pair $z,r$ we use a \texttt{GAP}~\cite{GAP4} script.
\begin{enumerate}[label=\arabic*.,topsep=0pt,itemsep=0.5ex]
	\item Find all groups $G$ of order $n=(z+r)^2+z+1$.
	\item If $G$ has an abelian index 2 subgroup, ignore it.
	\item Compute the list $U$ of orbit representatives of all inverse-closed sets $A$ of size $r$ such that $|A\cup AA|=r^2+1$.
	\item If $G$ admits an index 2 subgroup $H$, delete any sets from $U$ which do not satisfy Proposition~\ref{prop:mmcayidx2}.
	\item Compute the list $D$ of all inverse-free sets $B=\{a,b,(ab)^{-1}\}$ such that $|B\cup BB|=12$.
	\item Try (recursively) to extend each $S\in U$ by adding directed generators of order 3 or triples from $D$ until we have added $z$ generators.
\end{enumerate}

%----------------------------------------------
\section{Search results}
Results of the search on feasible orders up to 200 are in Table~\ref{tab:moore1}. For completeness the case $r=1$ is included.
As explained above, we know there is a unique Moore graph with $r=1$ for every $z\geq 1$, 
but these are Cayley only if $q=z+2$ is a prime power.
The algorithm reproduces all the known Cayley Moore graphs and confirms that there are no more examples below order 200.

We then continued the search for feasible orders up to 500. The results are in Table~\ref{tab:moore2}. 
The algorithm was unable to complete the search at order 486 due to the large numbers of groups and the increasing search space.
However, there are no more examples at any of the other feasible orders up to 485.

\begin{table}[ht]\centering
\begin{tabular}[t]{|cccc|}
\hline
$n$ & $r$ & $z$ & Graphs \\
\hline
18 & 3 & 1 & 1\\
40 & 3 & 3 & 0\\
54 & 3 & 4 & 0\\
84 & 7 & 2 & 0\\
88 & 3 & 6 & 0\\
108 & 3 & 7 & 2\\
150 & 7 & 5 & 0\\
154 & 3 & 9 & 0\\
180 & 3 & 10 & 0\\
\hline
\end{tabular}
\qquad
\begin{tabular}[t]{|cccc|}
\hline
$n$ & $r$ & $z$ & Graphs \\
\hline
6 & 1 & 1 & 1\\
12 & 1 & 2 & 1\\
20 & 1 & 3 & 1\\
30 & 1 & 4 & 0\\
42 & 1 & 5 & 1\\
56 & 1 & 6 & 1\\
72 & 1 & 7 & 1\\
90 & 1 & 8 & 0\\
110 & 1 & 9 & 1\\
132 & 1 & 10 & 0\\
156 & 1 & 11 & 1\\
182 & 1 & 12 & 0\\
\hline
\end{tabular}
\caption{Cayley Moore graphs up to order 200}
\label{tab:moore1}
\end{table}

\begin{table}\centering
\begin{tabular}[t]{|cccc|}
\hline
$n$ & $r$ & $z$ & Graphs \\
\hline
204 & 7 & 7 & 0\\
238 & 3 & 12 & 0\\
270 & 3 & 13 & 0\\
294 & 13 & 4 & 0\\
300 & 7 & 10 & 0\\
340 & 3 & 15 & 0\\
368 & 13 & 6 & 0\\
374 & 7 & 12 & 0\\
378 & 3 & 16 & 0\\
460 & 3 & 18 & 0\\
486 & 21 & 1 & ?\\
\hline
\end{tabular}
\qquad
\begin{tabular}[t]{|cccc|}
\hline
$n$ & $r$ & $z$ & Graphs \\
\hline
210 & 1 & 13 & 0\\
240 & 1 & 14 & 1\\
272 & 1 & 15 & 1\\
306 & 1 & 16 & 0\\
342 & 1 & 17 & 1\\
380 & 1 & 18 & 0\\
420 & 1 & 19 & 0\\
462 & 1 & 20 & 0\\
\hline
\end{tabular}
\caption{Cayley Moore graphs from order 200 to 500}
\label{tab:moore2}
\end{table}

We summarise these results as tabulated in Tables~\ref{tab:moore1} and~\ref{tab:moore2} in a theorem.
\begin{theorem}
Up to order 485, the only mixed Moore Cayley graphs of undirected degree $r$, directed degree $z$ and diameter 2 are as follows.
\begin{itemize}
	\item $r=1$ and $z\leq 20$ where $z+2$ is a prime power (Kautz graphs).
	\item $r=3$ and $z=1$ (Bos\'ak's graph).
	\item $r=3$ and $z=7$ (the two graphs of J\o rgensen).
\end{itemize}
\end{theorem}

%----------------------------------------------
\bibliographystyle{plain}

\end{document}